\documentclass[11pt]{article}
\usepackage [dvips]{graphics}
\usepackage[centertags]{amsmath}
\usepackage{amsfonts}
\usepackage{amssymb}
\usepackage{amsthm}
\usepackage{newlfont}
 \usepackage[]{titletoc}

\newlength{\defbaselineskip}
\newcommand{\setlinespacing}[1]%
           {\setlength{\baselineskip}{#1 \defbaselineskip}}

\theoremstyle{plain}
\newtheorem{thm}{Theorem}[section]
\newtheorem{cor}[thm]{Corollary}
\newtheorem{lem}[thm]{Lemma}
\newtheorem{prop}[thm]{Proposition}
\newtheorem{exam}[thm]{Example}
\newtheorem{rem}[thm]{Remark}

\newcommand{\C}{\mathbb{C}}

\makeatletter\@addtoreset{equation}{section} \makeatother
\begin{document}
\title {A tuple of multiplication operators defined by twisted holomorphic proper maps}%local inverse and proper map
\author{Pan Ma \quad  Hansong Huang}
\date{}
 \maketitle \noindent\textbf{Abstract:}
  This paper mainly concerns the von Neumann algebras
  induced by a tuple of multiplication operators on Bergman spaces
which arise essentially from    holomorphic proper maps  over higher
dimensional domains.
 We study the structures and abelian properties of the related von Neumann algebras,
 and  in interesting cases   they turns out to be tightly related
 to a Riemann manifold.
There is a close interplay between operator theory, geometry and
complex  analysis.   Many   examples are presented.
 % Local solution is defined   as a natural generalization of
%  local inverse, and one will see that how local solutions give rises to a local inverse.


%47A20; 46H25; 46C99.

\section{Introduction}
%Hartogs phenomenon
%
%$\Omega^*=\overline{\Omega}\circ$,  $\Phi^{-1}(\Phi(\Omega^*-\Omega))$,
%locally $L_a^2$-admissible.

~~~~ In this paper, $\mathbb{D}$ denotes the unit disk in the complex plane
$\mathbb{C}$. %   $\mathbb{B}_d$ denotes the unit ball in $\mathbb{C}^d.$
Let $\Omega$   denote a bounded domain in the complex space $\mathbb{C}^d$, and
$dA$ be the Lebesgue measure on $\Omega$.  The Bergman space $L_a^2(\Omega)$
is the Hilbert space consisting of  all  holomorphic functions over $\Omega$ which are square
 integrable with respect to the Lebesgue measure $dA$.
 For a holomorphic map $F:\Omega\to \C^d,$
  denote by $J F $   the determinant of the Jacobian of $F.$

For a bounded holomorphic function $ \phi  $ over   $\Omega$,
 let $M_{\phi ,\Omega}$ denote the multiplication operator with the symbol $\phi $   on  $L_a^2(\Omega)$, given by
$$M_{\phi ,\Omega} f=\phi  f, \, f\in L_a^2(\Omega).$$
In general, for a tuple $\Phi =\{\phi_j:1\leq j \leq n\}$, let $\{M_{\Phi  ,\Omega}\}'$
 denote the commutant of $\{M_{\phi_j }:1\leq j \leq n\}$,
 consisting of all bounded operators commuting with each operator $ M_{\phi_j,\Omega} (1\leq j \leq n )$.
  Here, we emphasize that  $M_{\Phi  ,\Omega}$ denotes a family
of multiplication operators rather than a single vector-valued multiplication operator.
Let $\mathcal{V}^*(\Phi ,\Omega) $ denote the von Neumann algebra
 $$\{ M_{\phi_j},\,  M_{\phi_j }^*:1\leq j \leq n \}',$$ which consists of all
 bounded operators on $L_a^2(\Omega)$ commuting with both $ M_{\phi_j }$ and $M_{\phi_j }^*$ for each $j$.
 It is   known that there is a close connection between orthogonal projections in $\mathcal{V}^*(\Phi ,\Omega) $ and
all joint reducing subspaces of $\{ M_{\phi_j }  :1\leq j \leq n
\}$. Precisely, the range of an orthogonal projection  in
$\mathcal{V}^*(\Phi ,\Omega) $ is exactly an joint reducing
subspaces of $\{ M_{\phi_j }  :1\leq j \leq n \}$,  and vice versa.
 We say that $\mathcal{V}^*(\Phi ,\Omega) $ is  trivial
 if $\mathcal{V}^*(\Phi ,\Omega) =\mathbb{C}I$.  This is equivalent to that
  $\{ M_{\phi_j }  :1\leq j \leq n \}$ has no nonzero joint reducing subspace
   other than  the whole space.
In single-variable case,  investigations on commutants and reducing subspaces of multiplication operators and
   von Neumann algebras induced  by those operators
  have been done in \cite{Cow1,Cow2,Cow3,CW,DPW,DSZ,GH1,GH2,GH3,GH4,SZZ,T1,T2}.
For  multi-variable case, one can refer to  \cite{DH,GW,HZ,Ti,WDH}.

This paper   consider a class of tuples $\Phi$ of holomorphic
functions and the von Neumann related to the corresponding
multiplication operators. It turns out that techniques of geometry,
complex analysis and operator theory are intrinsic in the study.

Suppose  $\Omega_1$ and $\Omega_2$ are two bounded
domains in $\mathbb{C}$, $\phi $ and  $\psi$ are holomorphic on
$ \Omega_1 $ and $ \Omega_2 $, respectively. Define
$$ \Upsilon_{\phi ,\psi}(z_1,z_2)=(\phi (z_1)+\psi_2(z_2),\phi (z_1)^2+ \psi(z_2)^2),\, z_1\in \Omega_1, z_2\in \Omega_2.$$
Define
$$\mathcal{S}_{\phi ,\psi}=\{(z,w)\in \Omega_1\times\Omega_2:
 z\not\in \phi^{-1}\circ \psi\Big(Z(\psi')\Big) ,\, z\not\in  \psi^{-1}\circ \phi\Big(Z( \phi')\Big)\} .$$
 Under some situations, $\mathcal{S}_{\phi,\psi}$ turns out to be a Riemann surface, and then
let $n(\phi ,\psi)$  denote the number of components of
$\mathcal{S}_{\phi,\psi}$.
 We have the dimension formula for $\mathcal{V}^*( \Upsilon_{\phi_1,\phi_2},\Omega_1\times \Omega_2)$.
\begin{thm} Suppose that $\phi_1$ and $\phi_2$ are holomorphic proper maps over
 bounded
domains $\Omega_1$ and $\Omega_2$ in $\mathbb{C}$, respectively. If $\phi_1(\Omega_1)=\phi_2(\Omega_2)$, then
$$\dim \mathcal{V}^*( \Upsilon_{\phi_1,\phi_2},\Omega_1\times \Omega_2)
    =   n(\phi_1,\phi_1) n(\phi_2,\phi_2)+n(\phi_1,\phi_2)^2
     .$$ In this case, $ \mathcal{V}^*( \Upsilon_{\phi_1,\phi_2},\Omega_1\times \Omega_2) $ is not
$*$-isomorphic to the  von Neumann algebra   \label{dim}
$ \mathcal{V}^*(\phi_1(z_1), \phi_2(z_2),\Omega_1\times
\Omega_2)= \mathcal{V}^*(\phi_1,\Omega_1 )\otimes \mathcal{V}^*(\phi_2,
\Omega_2) $.
\end{thm}
\noindent The condition $\phi_1(\Omega_1)=\phi_2(\Omega_2)$ can not
be weaken to
$\overline{\phi_1(\Omega_1)}=\overline{\phi_2(\Omega_2)}$, as
illustrated by Example \ref{ex}.

In Theorem \ref{dim}, if   both $\phi_1$ and $\phi_2$ are finite
Blaschke products, then  we get  something interesting. The following  shows that the
abelian property of  $ \mathcal{V}^*( \Upsilon_{B_1,B_2},\mathbb{D}^2) $
  heavily rests on  the connectedness of  the   Riemann surface $\mathcal{S}_{B_1,B_2}$.
 \begin{thm} Write $\Upsilon_{B_1,B_2} (z)=(B_1(z_1)+B_2(z_2),B_1(z_1)^2+ B_2(z_2)^2 ),$ where both $B_1$
 and $B_2$ are finite Blaschke products.
Then  \label{bla}
  $ \mathcal{V}^*(\Upsilon_{B_1,B_2},\mathbb{D}^2) $ is abelian if and only if  $\mathcal{S}_{B_1,B_2}$
is connected.
\end{thm}

As a special case, the von Neumann algebra $\mathcal{V}^*(z_1^k+z_2^l, z_1^{2k} z_2^{2l},
\mathbb{D}^2)$ is abelian if and only if  $GCD(k,l)=1 $ ( Example
\ref{kl}). Besides, if more than two
  finite Blaschke products are involved, a generalization of Theorem \ref{bla} fails,
see Example \ref{tri}.

For one-variable holomorphic functions $f$ and $g$, let
 $$ \Upsilon_{f,g}=( f(z_1)+g(z_2),f(z_1)^2+ g(z_2) ^2).$$  By contrast with Theorem \ref{dim}, the following is of interest.
\begin{prop}  \label{ffgg}Suppose that  $f$ and $g$ are holomorphic maps over
$\overline{\mathbb{D}}.$  If $\overline{f(\mathbb{D})} \neq
\overline{g (\mathbb{D})}$, then $\mathcal{V}^*( \Upsilon_{f,g},
\mathbb{D}^2)  $ is $*$-isomorphic to
$\mathcal{V}^*(f,\mathbb{D})\otimes \mathcal{V}^*(g,\mathbb{D}).$
Furthermore, $\mathcal{V}^*(\Upsilon_{f,g}, \mathbb{D}^2)  $  is
abelian.
\end{prop}

To   generalize the above result, we introduce a notion called local solution.
Throughout this paper, without   other explanation   we always assume that \emph{ $\Omega$ is a
bounded domain  in $\mathbb{C}^d$ such that the interior of
$\overline{\Omega}$ equals $\Omega$}.   For two holomorphic maps
$\Phi$ and $\Psi$ from $\Omega$ to $\mathbb{C}^d$, if there is   a
subdomain $\Delta$ of $\Omega $
 and   a holomorphic function
 $\rho$ over $\Delta$ such that
 $$\Psi(\rho(z))=\Phi(z),\, z\in \Delta,$$
 then $\rho$ is called  \emph{a  local solution}. If $\Phi$ and $\Psi$ are holomorphic proper maps over $\Omega$, one
can define a  complex manifold $\mathcal{S}_{\Phi,\Psi},$ which
equals $$  \{(z,w)\in \Omega : \Psi(w)=\Phi(z)\}$$ minus some small subset.
%This is inspired by the Riemann surface  $\mathcal{S}_{B,B} $ defined in \cite{GSZZ}, where $B$ is a finite Blaschke product.
% \textbf{It turns out that all related admissible local solutions have a close link
%with $\mathcal{S}_{\Phi,\Psi}.$}
We  also generalize the definition of a holomorphic map
$\Upsilon_{\Phi,\Psi}$.
 For  convenience, let $\Omega$ be a domain in $\mathbb{C}^2$ and
$$\Phi(z_1,z_2)=(\phi_1(z_1,z_2),\phi_2(z_1,z_2) ),\, \Psi(z_1,z_2)=(\psi_1(z_1,z_2),\psi_2(z_1,z_2) ),$$ where $ (z_1,z_2)\in \Omega.$
Write  $P(z)=( \sum_{j=1}^4 z_j, \sum_{j=1 }^4z_j^2,\cdots, \sum_{j=1}^4 z_j^4),$
 and $\Upsilon_{\Phi,\Psi}$ is defined by
$$\Upsilon_{\Phi,\Psi}(z_1,z_2,z_3,z_4) = P\circ \big(\Phi(z_1,z_2), \Psi(z_3,z_4)\big),\, (z_1,z_2,z_3,z_4)\in \Omega^2, $$
called the twisted map of $\Phi$ and $\Psi.$
Then we have the following.
\begin{thm} Suppose $\Phi$ and $\Psi $  are two holomorphic proper maps over $\Omega $  such that  \label{32}
$\overline{\Phi(\Omega)}\neq \overline{\Psi(\Omega)}$, and both $\Phi$ and $\Psi $  are holomorphic on $\overline{\Omega}$.
 Assume that $\Upsilon_{\Phi,\Psi}$ has no nontrivial  compatible equation.
Then $\mathcal{V}^*(\Upsilon_{\Phi,\Psi},\Omega^2)$  is $*$-isomorphic to
 $\mathcal{V}^*( \Phi, \Omega)\otimes \mathcal{V}^*( \Psi,\Omega)$.
 \end{thm}
\noindent
The condition  of $\Upsilon_{\Phi,\Psi}$ having no nontrivial  compatible equation is quite  geometric,
and such a condition is easy to check under many cases, see below (\ref{comp}). In addition,
Theorem \ref{32} still holds if $\Omega$ is a domain in $\mathbb{C}^d(d\geq 1).$

% Last but not the least,     the structure of
%$\mathcal{V}^*(\Upsilon_{\Phi,\Psi},\Omega^2)$ is determined by so called admissible
%local inverses of  $\Upsilon_{\Phi,\Psi}$.
% If in Theorem \ref{32} the condition $\overline{\Phi(\Omega)}\neq \overline{\Psi(\Omega)}$
% is replaced with $\Phi(\Omega)=\Psi(\Omega)$,
%  it turns out that some admissible local solutions for $\mathcal{S}_{\Phi,\Psi} $
%   and $\mathcal{S}_{\Psi,\Phi} $ give rise to admissible local inverses of  $\Upsilon_{\Phi,\Psi}$, see Theorem \ref{53}.

This paper is arranged as follows.

 In  Section \ref{sect2} we give some preliminaries. Section
\ref{sect3} gives the definition and studies some properties of local
solution, and provides the proof of Theorem \ref{dim}.
  Section 4 mainly  establishes Theorem \ref{bla} and discusses some  related examples.
    As is shown, there is an interplay between operator theory,
    geometry and complex analysis.
     Section \ref{sect5} aims to give a generalization of these results.
    We  focus  on two twisted holomorphic proper maps and present
the proof of   Theorem \ref{32}.

\section{\label{sect2}Some preliminaries}

This section gives some preliminaries, including proper map, zero variety and local inverse,
 which will be useful later.

Let $\Omega,\Omega_0,\Omega_1$ be   domains in $ \mathbb{C}^d$.
 A holomorphic function $\Psi: \Omega_0\to \Omega_1$ is called \emph{a proper map} if
 for any sequence $\{p_n\}$($p_n\in \Omega_0$) without limit point in $\Omega_0$,
$\{\Psi(p_n)\}$ has no limit point in $\Omega_1$.
% Equivalently,   each compact subset $K$ of $\Omega_0$, $\Psi^{-1}(K)$ is compact.
A holomorphic function $\Psi$ on $\Omega$ is called proper if
$\Psi(\Omega)$ is open and  $\Psi: \Omega \to \Psi (\Omega)$ is
proper. In particular, if $\Psi$ is holomorphic on
$\overline{\Omega}$, then $\Psi$ is proper on $\Omega$ if and only
if $\Psi(\Omega)$ is open and $$\Psi (\partial \Omega)\subseteq
\partial  \Psi(\Omega) .$$
For example, a holomorphic map $\Phi$ from $\mathbb{D}^d$ to
$\mathbb{D}^d$ is proper if and only if
 up to a permutation of coordinates $$\Phi(z)=(\phi_1(z_1),\cdots, \phi_d(z_d)),$$ where each
 $\phi_j(1\leq j \leq d)$ is a finite Blaschke product (\cite[Theorem 7.3.3]{Ru1}).

In general, a  holomorphic proper map is  open, which  follows
directly from \cite[Theorem 15.1.6]{Ru2}, see the following.
\begin{thm} \label{open} Suppose  $F: \Omega \to  \mathbb{C}^d$ is a holomorphic function
and for each $w\in \C^d$, $F^{-1}(w)$ is compact. Then $F$ is an
open map.
\end{thm}
Let $F:\Omega\to \mathbb{C}^d$ be a holomorphic map and  $Z$ denotes
the zeros of the determinant of the Jacobian of $F$. Then its image
$F(Z)$ is
 called \emph{the critical set} of $F$.  Each point in $F(Z)$ is called a critical value, and
each point in $F(\Omega)-F(Z)$ is called a regular point. A
holomorphic proper map is always an $m$-folds map, and its critical
set is  a zero variety (\cite[Theorem 15.1.9]{Ru2}), stated as
below.
\begin{thm} Given two domains $\Omega$ and $\Omega_0$ in $\mathbb{C}^d$, suppose $F:\Omega\to \Omega_0$ is a
holomorphic proper function. Let $\sharp (w)$ denote the number of
points in $F^{-1}(w)$ with $w\in \Omega_0$.  \label{variety} Then
the following hold:
\begin{itemize}
  \item[$(1)$] There is an integer $m$ such that $\sharp (w)=m$ for all regular values $w$ of $F$
  and  $\sharp (w')<m$ for all critical values $w'$ of $F$;
  \item [$(2)$] The  critical set of $F$ is a zero variety in $\Omega_0.$
\end{itemize}
\end{thm}
A subset $E$ of $\Omega$ is called \emph{ a zero variety }of
$\Omega$ if there is a nonconstant holomorphic function $f$ on
$\Omega$ such that $E=\{z\in \Omega | f(z)=0\}$.
A relatively closed subset $V$ of
$\Omega$  is called \emph{an (analytic) subvariety} of $\Omega$ if for
each point $w$ in $\Omega$ there is a neighborhood
 $\mathcal{N}$ of $w$ such that $V\cap \mathcal{N}$ equals the intersection of zeros of finitely many
holomorphic functions  over $\mathcal{N}$.

 An easier version of Remmert's Proper Mapping Theorem reads as follows, see \cite[p. 65]{Ch} or \cite{Re1,Re2}. % Remmert Proper Mapping theorem
  \begin{thm} \label{remmert}If $f:\Omega_0\to \Omega_1$ is a holomorphic proper map  and $\mathcal{Z}$ is a
subvariety of $\Omega_0$, then $f(\mathcal{Z})$ is a subvariety of
 $\Omega_1.$
\end{thm}

  We need some notions about analytic
continuation (\cite[Chapter 16]{Ru3}).  A function element  is an
ordered pair $(f,D)$, where $D$ is an open ball in $\mathbb{C}^d$
and $f$ is a holomorphic function on $D$. Two  function elements
$(f_0,D_0)$  and $(f_1,D_1)$ are  called direct continuations if
$D_0\cap D_1$ is not empty and $f_0 =f_1 $ holds on $D_0\cap D_1$. A
curve is
  a continuous map from $[0,1]$ into  $
\mathbb{C}^d$.
 For a function element $(f_0,D_0)$ and  a curve    $\gamma$   with $\gamma(0)\in D_0$,
if there is  a partition of  $[0,1]$:
     $$0=s_0<s_1<\cdots <s_n=1$$
 and function elements $(f_j,D_j)(0\leq j \leq n)$ such that
   \begin{itemize}
  \item [1.]  $(f_j,D_j)$  and $(f_{j+1},D_{j+1})$
are direct continuation for all $j$ with\linebreak $ 0\leq j\leq n-1
$;
  \item[2.] $\gamma [s_j,s_{j+1}]\subseteq D_j(0\leq j\leq n-1)$ and $\gamma(1)\in D_n$,
       \end{itemize}then   $(f_n, D_n)$ is called \emph{an analytic continuation of $(f_0, D_0)$ along
     $\gamma$} ; and $(f_0, D_0)$
 is called to \emph{admit} an analytic continuation along  $\gamma$.
 In this case, we write   $f_0\sim  f_n$. Clearly, $\sim$ defines an equivalence and we write
$[f]$ for \emph{the equivalent class} of $f$.

As follows, we will give a generalization of    local inverse. For
convenience,  assume  both $\Phi$ and $\Psi$ are   holomorphic
  maps  from $\overline{\Omega} $ to $ \mathbb{C}^d.$
   Write
 $$ Z(J\Phi) =\mathcal{Z}_\Phi \quad \mathrm{and} \quad  Z(J\Psi) =\mathcal{Z}_\Psi.$$
 Let    $$\mathcal{S}_{\Phi,\Psi} =\{(z,w)\in \Omega : \Psi(w)=\Phi(z),
   z\not\in \Phi^{-1}\big(\overline{\Psi ( \mathcal{Z}_\Psi )})\}.$$
and
$$\mathcal{S}_{\Psi,\Phi} =
\{(z,w)\in \Omega : \Phi(w)=\Psi(z),
   z\not\in \Psi^{-1}\big( \overline{\Phi (\mathcal{Z}_\Phi )})\}. $$
It can happen that $\mathcal{S}_{\Phi,\Psi} $ or
$\mathcal{S}_{\Psi,\Phi} $  is empty,  but  in interesting cases
they are Riemann manifolds.
    If there is   a subdomain $\Delta$ of $\Omega $
 and   a holomorphic function
 $\rho$ over $\Delta$ such that
 $$\Psi(\rho(z))=\Phi(z),\, z\in \Delta,$$
 then $\rho$ is called \emph{a  local solution} for   $\mathcal{S}_{\Phi,\Psi}$, % the pair $(\Phi,\Psi)$, %(or ),
   denoted by $$\rho\in \Psi^{-1}\circ \Phi.$$
 In particular, if $\Phi=\Psi,$ then $\rho$ is called \emph{a local inverse} of $\Phi$ \cite{T1} and
 we rewrite $S_\Phi$ for $\mathcal{S}_{\Phi,\Phi}$.
Following \cite{HZ}, a local solution $\rho$   for
$\mathcal{S}_{\Phi,\Psi}$ is called
\emph{admissible} if %there is a relatively closed  subset $A$ of $\Omega $ such that $A$ is locally contained in a zero variety   % Lebesgue measure zero
  for each curve $\gamma$ in $\Omega- \Phi^{-1}\big( \overline{\Psi (\mathcal{Z}_\Psi )})$, $\rho$ admits analytic
continuation with values in $\Omega .$ In this case, we say $\rho$
is admissible with respect to \linebreak
 $ \Phi^{-1}\big(\overline{\Psi ( \mathcal{Z}_\Psi )})$.  %\textbf{As done in} \cite[Proposition ??]{HZ},
 It can be shown that
$\Omega- \Phi^{-1}\big(\overline{\Psi ( \mathcal{Z}_\Psi)} )$ is
connected if both $\Phi$ and $\Psi $ are holomorphic on
$\overline{\Omega} $.
\begin{rem}By assuming that both $\Phi$ and $\Psi$ are   holomorphic
  maps  from $\overline{\Omega} $ to $ \mathbb{C}^d,$ we define $\mathcal{S}_{\Phi,\Psi}$,
$\mathcal{S}_{\Psi,\Phi}$ and admissible local solution. Now
  let $\Phi$ and $\Psi$ be holomorphic proper maps defined on $\Omega$ and $$\Phi(\Omega)=\Psi(\Omega)
  .$$
In this case, %using Theorem \ref{remmert} one can show  that
 $$\Omega- \Phi^{-1}\big(\Psi ( \mathcal{Z}_\Psi ))$$
is also connected.  Furthermore, $$\Psi^{-1}\big(\overline{\Phi (
\mathcal{Z}_\Phi )}) =\Psi^{-1}\big(\Phi ( \mathcal{Z}_\Phi )),$$
and
$$\Phi^{-1}\big(\overline{\Psi ( \mathcal{Z}_\Psi )}) =\Phi^{-1}\big(\Psi ( \mathcal{Z}_\Psi )).$$
To define  $\mathcal{S}_{\Phi,\Psi}$, $\mathcal{S}_{\Psi,\Phi}$ and
etc.,
 we must replace $\overline{\Phi(\mathcal{Z}_\Phi)}$ and $\overline{\Psi(\mathcal{Z}_\Psi)}$
  with $\Phi(\mathcal{Z}_\Phi)$ and $\Psi(\mathcal{Z}_\Psi)$, respectively.
\end{rem}

%It is shown in \cite{HZ}  that for a holomorphic proper map $\Phi$ on $\Omega $, $
%\mathcal{V}^*(\Phi,\Omega) $ is nontrivial if and only if $\Phi$ is not biholomorphic.
\vskip2mm
Given an admissible  local inverse $\rho$ of $\Phi$, define
$$\mathcal{E}_{[\rho]}h(z)=\sum_{\sigma\in [\rho]}h\circ \sigma(z) J\sigma(z),
 h\in L_a^2(\Omega), z\in \Omega-\Phi^{-1}\big(\overline{\Phi ( \mathcal{Z}_\Phi )}).$$
 The following results come from \cite{HZ}.
 \begin{thm} Suppose $\Phi :\Omega\to \C^d$ is  holomorphic  on $\overline{\Omega}$ and
 the image of  $\Phi $ contains an interior point.
 Then $\dim \mathcal{V}^*(\Phi,\Omega)<\infty,$ and $\mathcal{V}^*(\Phi,\Omega)$ is generated by   \label{ormain}
  $\mathcal{E}_{[\rho]}$,  where $\rho$ run over admissible local inverses  of
  $\Phi.$
 \end{thm}
    \begin{thm}Let $\Omega$ and $\Omega_0$ be bounded domains in $\mathbb{C}^d$.
    Suppose \linebreak $\Phi   :\Omega\to \Omega_0$ is  a holomorphic proper
    map.
 Then $\mathcal{V}^*(\Phi ,\Omega)$ is generated by   \label{proper}
  $\mathcal{E}_{[\rho]}$, where $\rho$ are local inverses  of $\Phi .$
   In particular, the dimension of $\mathcal{V}^*(\Phi ,\Omega)$ equals the number of
   components of $\mathcal{S}_{\Phi }$. % and it also equals the number of    equivalent classes of   local inverses of $\Phi $.
    \end{thm}% If in addition, all such $\rho$ are holomorphic in $\Omega,$ then

% Also we need a result from operator algebra. It determines the structure of   finite dimensional von Neumann
%algebras  \cite[Theorem III.1.2]{Da}.
%\begin{thm}
%Assume that $\mathcal{A}$ is a  finite dimensional von Neuamm
%\label{operator}
%algebra on a Hilbert space $H$. Then $\mathcal{A}$ is $ *$-isomorphic to %$C^*$
%$\bigoplus_{k=1}^r M_{n_k}(\mathbb{C}).$ \end{thm}

\section{\label{sect3}Properties of local solutions}
 This section gives some properties of local solutions and the proof
of Theorem \ref{dim} is presented.

For a domain $\Omega$ in $\mathbb{C}^d,$ first assume that both
  $\Phi$ and $\Psi$ are holomorphic on $\overline{\Omega}.$
  In this case,   $\mathcal{S}_{\Phi,\Psi}$ is a nonempty set.
For two local solutions $\rho$ and $\sigma$ for $\mathcal{S}_{\Phi,\Psi}$,
 if $\rho$ is an analytic continuation of $\sigma$, then their images lie in a same component of
 $\mathcal{S}_{\Phi,\Psi}$, and vice versa. Therefore, \emph{the number of   classes of  local solutions equals the number
 of components of $\mathcal{S}_{\Phi,\Psi}$.}

  For simplicity, we modify
$\mathcal{S}_{\Phi,\Psi}$ a bit by setting
 $$\mathcal{S}_{\Phi,\Psi} =\Big\{(z,w)\in \Omega : \Psi(w)=\Phi(z),
 z\not\in \Phi^{-1}\big(\overline{\Psi ( \mathcal{Z}_\Psi )}) ,
 \, z\not\in  \Psi^{-1}\big(\overline{\Phi ( \mathcal{Z}_\Phi )})\Big\}.$$
 That is, \begin{equation}
 \mathcal{S}_{\Phi,\Psi} =\Big\{(z,w)\in \Omega : \Psi(w)=\Phi(z)\}-\Gamma_{\Phi,\Psi},  \label{new}
 \end{equation}
 where $$\Gamma_{\Phi,\Psi}=\{(z,w)\in \Omega : \Psi(w)=\Phi(z),
 z \in \Phi^{-1}\big(\overline{\Psi ( \mathcal{Z}_\Psi )})
 \cup   \Psi^{-1}\big(\overline{\Phi ( \mathcal{Z}_\Phi )}) \Big\} .$$
 Note that the numbers of components of  both original  $\mathcal{S}_{\Phi,\Psi}$ and
$\mathcal{S}_{\Psi,\Phi}$ remains invariant after we remove only
a special set of complex codimension $1$ away from them.
 This can be seen as follows. By
 \cite[Theorem 14.4.9]{Ru2},  a zero
variety of a domain in  $\mathbb{C}^d$ can be represented as a union
of compact sets  $L_n$ whose Hausdorff measures
$h_{2d-2}(L_n)<\infty.$ %and the same is true for  a  union of countably many zero varieties.
   By \cite[Theorem 14.4.5]{Ru2}, for a connected domain $\mathcal{U}$ in $\mathbb{R}^{2d}$, if
   a relatively closed set  $G$ of  $\mathcal{U}$  can be written as the union of countably many
   compact sets $K_n$ whose Hausdorff measure $h_t(K_n) <\infty$ for some $ t\in (0, 2d-1),$
   then $\mathcal{U}-G$ is connected. Note that  $\Psi^{-1}\big(\overline{\Phi ( \mathcal{Z}_\Phi )})$
  can be presented as a union of countably many compact sets   whose Hausdorff
  dimensions are not larger than $2d-2$,
   as well as $ \overline{\Phi ( \mathcal{Z}_\Phi )}$. Letting $G=\Psi^{-1}\big(\overline{\Phi ( \mathcal{Z}_\Phi )})$
    we have
   $$\mathcal{O}-\Psi^{-1} \big(\overline{\Phi ( \mathcal{Z}_\Phi )})$$ is
   connected for any subdomain $\mathcal{O}$ of $\Omega$.
  Note that   $\mathcal{S}_{\Phi,\Psi}$
 consists of disjoint connected open components. Since
   the complex manifold $\mathcal{S}_{\Phi,\Psi}$ defined by
   (\ref{new})
    is exactly the original one minus $\Psi^{-1}\big(\overline{\Phi ( \mathcal{Z}_\Phi )}\big)$,
   they have the same number of components.

Note that $\mathcal{S}_{\Phi,\Psi}$ and $\mathcal{S}_{\Psi,\Phi}$
are equal now up to a permutation of coordinates, and they have the
same number of components. Thus the numbers of equivalent classes of
local solutions for $\mathcal{S}_{\Phi,\Psi}$ and
$\mathcal{S}_{\Psi,\Phi}$ are exactly equal. Letting   $n(\Phi,\Psi) $
denote the number of components of  $\mathcal{S}_{\Phi,\Psi}$,
  we have the following.
\begin{prop}\label{number} Suppose one of the following holds:
\begin{itemize}
  \item[$(i)$]  both $\Phi:\Omega\to \Phi(\Omega)$ and $\Psi:\Omega\to \Psi(\Omega)$
   are holomorphic proper maps and   $ \Phi(\Omega) = \Psi(\Omega);$
  \item [$(ii)$]  both $\Phi$ and $\Psi$ are  holomorphic over
$\overline{\Omega}$ and their  images in $\mathbb{C}^d$ have an
interior point.
\end{itemize}
Then $\mathcal{S}_{\Phi,\Psi}$ and $\mathcal{S}_{\Psi,\Phi}$ have
the same number of components;  that is, \linebreak  $n(\Phi,\Psi)=n(\Phi,\Psi) $.
% there  are equivalent classes of local solutions for $\mathcal{S}_{\Phi,\Psi}$
%$($or   $\mathcal{S}_{\Psi,\Phi})$.
\end{prop}
 Under Condition (i),
 the  local solutions for $\mathcal{S}_{\Phi,\Psi}$ turn out to be admissible.
The special case of $\Phi =\Psi$ is discussed in
the proof of  \cite[Theorem 1.4]{HZ}. %  ????
\begin{lem} Suppose that $\Phi $ and $\Psi$ are   holomorphic proper  maps on $\Omega $
with  same images.  \label{admiss}
 Then each local solution  $\rho$ for $\mathcal{S}_{\Phi,\Psi}$ is admissible in $\Omega.$
    \end{lem}
\begin{proof}Suppose that $\Phi $ and $\Psi$ are  holomorphic proper    maps on $\Omega $ with the same images.  Write
$$A= \Phi^{-1}\big(\Psi ( \mathcal{Z}_\Psi )\big).$$ %$$  \bigcup   \Psi^{-1}\circ \Phi\Big(Z(J\Phi)\Big)$$
 Since   $\Psi$ are    proper,  by Theorem \ref{variety}
   $\Psi ( \mathcal{Z}_\Psi ) $ is a zero variety. Then
      $\Psi ( \mathcal{Z}_\Psi )$ is  relatively closed in $\Psi(\Omega),$
  and thus  $A$ is  relatively closed in $ \Omega $.

For each curve $\gamma$ in $\Omega-A$, write $z_0=\gamma(0).$
Given a local solution $\rho$ satisfying $$\Psi(\rho(z_0))=\Phi(z_0),$$ we
will prove that $\rho$ admits analytic continuation along $\gamma.$
To see
 this, note that $\Phi$ and $\Psi$ has  same images. For each point $w$ on $\gamma$,
$\Phi(w)\in \Psi(\Omega).$
 Since
$\gamma \subseteq \Omega-A$,
$$ \Phi (\gamma) \cap \Psi ( \mathcal{Z}_\Psi ) =\emptyset .$$ By Theorem \ref{variety},
 there is a constant interger $n$ depending only on $\Psi$ so that
$\Psi^{-1}(\Phi(w))$ has exactly $n$  distinct  points. Furthermore,
there is an open ball $U_w$ centered at $w$ and $n$  holomorphic
maps $\rho_1^w, \cdots, \rho_n^w$ over $U_w$ satisfying
$$\Psi(\rho_j^w(z))=\Phi(z),\,z\in U_w,\, 1\leq j \leq n.$$
Since  $\gamma$ is compact, by Henie-Borel's theorem there
 are finitely many such balls $U_w$ whose union contains $\gamma.$ Then
  it is straightforward to show that  all $\rho_j^{z_0}$ $(1\leq j \leq n)$
admit analytic continuation along $\gamma.$ Since one of
$\rho_j^{z_0}$ \linebreak $(1\leq j \leq n)$ is the direct
continuation of $\rho,$ $\rho$  admits analytic continuation along
$\gamma.$ The
proof is complete.
\end{proof}
We are ready to  give the proof of Theorem \ref{dim}.
\vskip2mm
\noindent  \textbf{Proof of Theorem \ref{dim}.}
Suppose that both $\phi_1$ and $\phi_2$ are holomorphic proper maps over
 bounded domains $\Omega_1$ and $\Omega_2$ in $\mathbb{C}$, respectively, and $\phi_1(\Omega_1)=\phi_2(\Omega_2)$.
Let $$\Omega=\phi_1(\Omega_1)=\phi_2(\Omega_2).$$ By Theorem \ref{open} one can show that $(z_1+z_2,z_1^2+z_2^2)$ is
 an open map and by analysis    it is a proper map on $\Omega\times \Omega$.
 Then as a composition of   $(z_1+z_2,z_1^2+z_2^2)$  and $(\phi_1(z_1),\phi_2(z_2))$, $ \Upsilon_{\phi_1 ,\phi_2}$ is
  a holomorphic proper map on
 $\Omega_1\times \Omega_2, $ and by Theorem \ref{proper} the von Neumann algebra $\mathcal{V}^*(\Upsilon_{\phi_1 ,\phi_2},\Omega_1\times \Omega_2)$
 is generated by $\mathcal{E}_{\rho}$, where $\rho$ run over   local inverses  of $\Upsilon_{\phi_1 ,\phi_2}$.
 Note by Lemma \ref{admiss} all these $\rho$ are necessarily admissible.

Below we will determine the local inverse of  $\Upsilon_{\phi_1 ,\phi_2}$.
Observe that
\begin{equation}(\lambda_1+\lambda_2,\lambda_1^2+\lambda^2_2)=(\mu_1+\mu_2,\mu_1^2+\mu^2_2) \label{test}\end{equation}
is equivalent to $$(\lambda_1+\lambda_2,\lambda_1  \lambda_2)=(\mu_1+\mu_2,\mu_1 \mu_2),$$
Then $(\lambda_1,\lambda_2)$ and $(\mu_1,\mu_2)$ are the same
zeros of a polynomial $p$  counting multiplicity, with
$p(x)= x^2+ (\lambda_1+\lambda_2) x+\lambda_1  \lambda_2 .$
Thus the solutions for (\ref{test}) are
$$(\lambda_1,\lambda_2)=(\mu_1,\mu_2)$$
and
$$(\lambda_1,\lambda_2)=(\mu_2,\mu_1).$$
By this observation, solving
 $$ \Upsilon_{\phi_1,\phi_2} (w_1,w_2)=\Upsilon_{\phi_1,\phi_2} (z_1,z_2) $$
 is equivalent to solving
 \[    \left\{\begin{array}{cc}
 \phi_1(w_1)=\phi_1(z_1), \\
   \phi_2(w_2)=\phi_2(z_2),
\end{array}\right.
\]
and  \[    \left\{\begin{array}{cc}
 \phi_1(w_1)=\phi_2(z_2), \\
   \phi_2(w_2)=\phi_1(z_1).
\end{array}\right.
\]
Then we have  either
\begin{equation} \label{ge}
(w_1,w_2)=(\sigma_1(z_1),\sigma_2(z_2)),\, \sigma_1\in \phi_1^{-1}\circ \phi_1, \,
\sigma_2 \in   \phi_2^{-1}\circ \phi_2,
\end{equation}
or
\begin{equation*}
 (w_1,w_2)=( \tau_1(z_2),  \tau_2(z_1) ),\,    \tau_1 \in \phi_1^{-1}\circ \phi_2,\,
   \sigma_2\in \tau_2^{-1}\circ \phi_1.
 \end{equation*}
 Hence by Proposition \ref{number} $\Upsilon_{\phi_1,\phi_2}$
 has exactly $n(\phi_1,\phi_1) n(\phi_2,\phi_2)+n(\phi_1,\phi_2)^2$
 equivalent classes for admissible local inverses.
Since $\mathcal{V}^*( \Upsilon_{\phi_1,\phi_2},\Omega_1\times \Omega_2)$ is generated by
$\mathcal{E}_{\rho}$ where $\rho$ are admissible local inverses of  $\Upsilon_{\phi_1,\phi_2}$,
$$\dim \mathcal{V}^*( \Upsilon_{\phi_1,\phi_2},\Omega_1\times \Omega_2)
    =   n(\phi_1,\phi_1) n(\phi_2,\phi_2)+n(\phi_1,\phi_2)^2.$$
Since $\dim \mathcal{V}^*(\phi_j,\Omega_j )=n(\phi_j,\phi_j)$, $j=1,2$,
$$\dim \mathcal{V}^*(\phi_1,\Omega_1 )\otimes \mathcal{V}^*(\phi_2,
\Omega_2)=n(\phi_1,\phi_1) n(\phi_2,\phi_2)< \dim \mathcal{V}^*( \Upsilon_{\phi_1,\phi_2},\Omega_1\times \Omega_2).$$
 Therefore, $ \mathcal{V}^*( \Upsilon_{\phi_1,\phi_2},\Omega_1\times \Omega_2) $ is not
$*$-isomorphic to the  von Neumann algebra
$\mathcal{V}^*(\phi_1,\Omega_1 )\otimes \mathcal{V}^*(\phi_2,
\Omega_2).$
\vskip1mm
Besides, the map $(\phi_1(z_1),\phi_2(z_2))$ is a proper map whose local inverses are exactly of the form
(\ref{ge}), and by Theorem \ref{proper} $$\mathcal{V}^*(\phi_1(z_1),\phi_2(z_2),\Omega_1 \times
\Omega_2)=\mathcal{V}^*(\phi_1,\Omega_1 )\otimes \mathcal{V}^*(\phi_2,
\Omega_2).$$
The proof of Theorem \ref{dim} is finished.
$\hfill \square$

 \vskip2mm
The following example shows that $\phi_1(\Omega_1)=\phi_2(\Omega_2) $ is sharp  in the sense that it
can not be replaced with
$\overline{\phi_1(\Omega_1)}=\overline{\phi_2(\Omega_2)} .$

\begin{exam}\label{ex}Put $\Omega=\mathbb{D}-[-1,0]$. Write  $f(z)=z, z\in \mathbb{D}$ and
$g$ is the restriction of $f$ on $\Omega$.  Obviously, $f$and $g$ are  proper maps  on $\mathbb{D}$ and $\Omega$
respectively.
 Set
 $$\Upsilon_{f,g}(z_1,z_2)=(z_1+z_2,z_1^2+z_2^2),z_1\in \mathbb{D},z_2\in \Omega.$$
We will prove that $$ \mathcal{V}^*(\Upsilon_{f,g}, \mathbb{D}\times
\Omega)=\mathbb{C}I;$$ equivalently, $ \mathcal{V}^*(\Upsilon_{f,g},
\mathbb{D}\times \Omega)$ is $*$-isomorphic to $\mathcal{V}^*(f,
\mathbb{D} )\otimes \mathcal{V}^*(g, \Omega ).$ For this, letting $\rho(z_1,z_2)=(z_2,z_1),$ we have that each operator
 $S$ in $\mathcal{V}^*(\Phi_{f,g}, \mathbb{D}\times \Omega)$ is of the
 form
 $$Sh(z_1,z_2) =c_1h(z_1,z_2) + c_2 h\circ \rho(z_1,z_2), (z_1,z_2)\in \mathbb{D}\times \Omega.$$
If $ \mathcal{V}^*(\Upsilon_{f,g}, \mathbb{D}\times
\Omega)\neq \mathbb{C}I,$ $h\mapsto
h\circ \rho$ defines a bounded linear operator on
$L_a^2(\mathbb{D}\times \Omega)$,  and it  maps $L_a^2(\mathbb{D})\otimes
L_a^2(\Omega)$ to $L_a^2(\mathbb{D}\times \Omega)$. By the form of
$\rho,$ it is direct to check that each function in
$L_a^2(\Omega)$ extends holomorphically to a function in
$L_a^2(\mathbb{D})$. But this can not be true because  $\ln z_1 \in
L_a^2(\Omega)$ but $\ln
 z_1  \not\in
L_a^2( \mathbb{D})$ as  $\ln z_1$ can not be extended to an holomorphic function over $ \mathbb{D} $.
\end{exam}

Suppose that $\Phi $ and $\Psi$ are   holomorphic proper  maps on $\Omega $
with  same images.  Then the closure of the range
of an admissible local solution and all its continuations equals
that of $\Omega.$ To be precise, let $\sigma$ be an  admissible
local solution for $\mathcal{S}_{\Phi,\Psi}$ and $[\sigma]$ denotes
the equivalent class of $\sigma,$  Image $ [\sigma]$ denotes the
union of all images of local solutions in the equivalent class of
$\sigma.$  By definition we have $$\mathrm{Image} \,[\sigma] \subseteq
\Omega.$$ Then the inverse $\sigma^-$ of $\sigma$ is a local
solution of $\mathcal{S}_{\Psi,\Phi}$.  By Lemma \ref{admiss}, both
$\sigma$ and $\sigma^-$ are admissible with respect to the set
$\mathcal{E}$   defined by
$$\mathcal{E} = \Phi^{-1}(\Psi ( \mathcal{Z}_\Psi ))
\bigcup   \Psi^{-1}(\Phi (\mathcal{Z}_\Phi ).$$ We address that
$\Omega-\mathcal{E}$ is connected. Since $\sigma^-$ or its
continuation is well defined on each given  point of $\Omega-
\mathcal{E}$,  the union of  the images of $\sigma$ and all its
continuation contains
 $\Omega- \mathcal{E}$. That is,
$$\Omega-\mathcal{E} \subseteq \mathrm{Image}\, [\sigma] ,$$
forcing
\begin{equation*}  \Omega-\mathcal{E} \subseteq \mathrm{Image}\, [\sigma] \subseteq   \Omega.  \end{equation*}
Note that $\mathcal{E}$ is  relatively closed subset of $\Omega$
 with zero Lebesgue measure.   We get the following.
 \begin{lem} If both $\Phi$ and $\Psi$ are holomorphic proper maps on $\Omega$ with same images, then
for each  local solution $\sigma$ for  $\mathcal{S}_{\Phi,\Psi}$
\begin{equation} \overline{ \mathrm{Image}\, [\sigma]} =  \overline{ \Omega}.
 \label{subset}  \end{equation}
 \end{lem}
\noindent In particular, in the case of $\Phi=\Psi,$ (\ref{subset}) holds for each
  local inverse $\sigma$ of a holomorphic proper
 map $\Phi$ over $\Omega.$
 \begin{rem}
If  $\Phi$ is holomorphic  over $\overline{\Omega}$
and
 $\sigma$ is an   admissible local inverse of  $\Phi$, then \label{eq}
(\ref{subset})  still holds. The reasoning is  similar  to the above discussion.
 \end{rem}
% and the fact that the family of admissible local inverses of $\Phi$
%is closed under composition \cite{T1}.

\section{ \label{sectix} Twisted finite Blaschke products} %Intertwined
  This section  mainly establishes  Theorem \ref{bla} and gives some examples for it.
 As we will  see, there is an interplay between operator theory and
geometry of Riemann manifold.

% \begin{thm} Write $\Phi_{B_1,B_2} (z)=(B_1(z_1)+B_2(z_2),B_1(z_1) B_2(z_2) ).$
%Then $\Phi_{B_1,B_2}$ is a holomorphic proper map on  $\mathbb{D}^2.$ Furthermore, % \label{bla}
%  $ \mathcal{V}^*( \Phi_{B_1,B_2},\mathbb{D}^2) $ is abelian if and only if  $\mathcal{S}_{B_1,B_2}$
%is connected.
%\end{thm} delete.
We first provide the proof of Theorem \ref{bla}.

 \noindent \textbf{Proof of Theorem \ref{bla}.}  Let $B_1$ and $B_2$
 be finite Blaschke products and write
 $$\Upsilon_{B_1,B_2} (z)=(B_1(z_1)+B_2(z_2),B_1(z_1) ^2+B_2(z_2)^2 ),\, (z_1,z_2) \in \mathbb{D}^2.$$
   Since $\Upsilon_{B_1,B_2} $ is the composition of two  holomorphic proper maps  \linebreak $ (z_1+z_2,z_1^2+z_2^2)$ and
$(B_1(z_1),B_2(z_2))$, $\Upsilon_{B_1,B_2} $ is a holomorphic proper map
on  $\mathbb{D}^2.$

Below we will show that $ \mathcal{V}^*(
\Upsilon_{B_1,B_2},\mathbb{D}^2) $ is abelian if and only if
$\mathcal{S}_{B_1,B_2}$ is connected. To see this, by Theorem
\ref{ormain} studying $ \mathcal{V}^*( \Upsilon_{B_1,B_2},\mathbb{D}^2)
$ reduces to studying admissible local inverses of $
\Upsilon_{B_1,B_2}.$ For this, write
$$\Upsilon_{B_1,B_2}(w)= \Upsilon_{B_1,B_2}(z), w,z\in \mathbb{D}^2.$$ Then
by discussion in the proof of Theorem \ref{dim}, we get
  either
\begin{equation}
(w_1,w_2)=(\rho(z_1),\sigma(z_2)),\, \rho\in B_1^{-1}\circ B_1, \,
\sigma \in   B_2^{-1}\circ B_2, \label{loc1}
\end{equation}
or
\begin{equation}
 (w_1,w_2)=(\zeta(z_2), \eta(z_1) ),\,   \zeta \in B_1^{-1}\circ B_2,\, \eta\in  B_2^{-1}\circ B_1.  \label{loc2}
 \end{equation}
 By Lemma \ref{admiss}  all   $\rho $, $ \sigma  $, $\zeta$  and $\eta$  are admissible,
and the   local inverses  of
$\Upsilon_{B_1,B_2}$ defined in (\ref{loc1}) and (\ref{loc2}) are also admissible.% Observe that in (\ref{loc2}) a
%local solution for $\mathcal{S}_{B_1,B_2}$ and another for
%$\mathcal{S}_{B_2,B_1}$ give rise to a local inverse of the map $
%\Upsilon_{B_1,B_2}.$ Thus, two single-variable local solutions induce a
%two-variable local inverse.

 To discuss
the abelian property of $ \mathcal{V}^*(
\Upsilon_{B_1,B_2},\mathbb{D}^2) $, we must study whether their
equivalent class commute with each other under composition. This
 relies on  whether $([\rho] \circ [\zeta](z_2), [\sigma]
\circ [\eta] (z_1))$ equals
$$([\zeta]\circ [\sigma] (z_2), [\eta] \circ [\rho] (z_1)).$$
Before continuing, one must determine what is the composition of two
equivalent classes \cite{GH3}. Observe that for any  local inverses
$[\tau_1]$ and $[\tau_2]$,
  $\mathcal{E}_{[\tau_1]} \mathcal{E}_{[\tau_2]}$ has the form $$\sum_{j}\mathcal{E}_{[\sigma_j]},$$
  where the sum is finite and some $\sigma_j$ may lie in the same class; and
     we define \emph{the composition} $$[\tau_1]\circ [\tau_2]$$ to be a formal sum
  $\sum_{j} [\sigma_j] $. Thus   $$\mathcal{E}_{[\tau_1]} \mathcal{E}_{[\tau_2]}=
    \mathcal{E}_{[\tau_2]} \mathcal{E}_{[\tau_1]} $$
    if and only if  $[\tau_1]\circ [\tau_2]=[\tau_2]\circ [\tau_1]$.
       The  formal sum of $k$ same equivalent classes  $ [\sigma]$ is denoted by $k  [\sigma]$.

 Since a finite Blaschke product has no critical point on the unit
circle  $\mathbb{T}$,
 it is conformal on    $\mathbb{T}$.
 Thus for $j,k=1,2$ the local solutions for $\mathcal{S}_{B_j,B_k} $  are holomorphic on a neighborhood
 of each point on    $\mathbb{T}$. For these local solutions, we   just
 discuss  their behavior  on the unit circle    $\mathbb{T}$.
 Suppose order $B_1=m$ and order $B_2=n$. Let  $a_1, \cdots,  a_m$ be $m$  distinct points on
 $\mathbb{T}$ and $b_1, \cdots,  b_n$ be $n$  distinct points on
 $\mathbb{T}$, both  in anti-clockwise direction and satisfying
\begin{equation}B_1(a_j)=B_2(b_k)=1,\, 1\leq j \leq m, \,  1\leq k \leq n.  \label{bb1}
\end{equation}

 Suppose $S_{B_1,B_2}$ is not connected. That is, $S_{B_1,B_2}$ has more than one component.
   %Note that the image  $[\zeta](a_1) $ determines $[\zeta]$,
   Let $[\zeta](a_1) $ denote the set   of all
 $\widetilde{\zeta}(a_1)$ as  $\widetilde{\zeta}$ run over all analytic continuations along  loops
  in $\mathbb{T}$ beginning at $a_1.$
 Writing
 $$[\zeta](a_1)=\{b_k:k\in \Lambda\},$$
we have
$$[\zeta](a_1)\neq \{b_1, \cdots,  b_n\}.$$
 Thus there is at least a  local solution $\eta$
 of $S_{B_1,B_2}$  such that $\eta (a_1)\not\in [\zeta](a_1).$
 Write $$\eta (a_1)=b_{j_0}.$$  By conformal property of $B_1$ and $B_2 $ on $\mathbb{T}$,
   local solutions for $S_{B_1,B_2}$ (or $S_{B_2,B_1}$) admit continuation along any curve  in $\mathbb{T}$.
   In particular, by (\ref{bb1})
 there is an $a_{j}$ such that $\zeta^{-1}(b_{j_0})=a_j$, forcing  $$\zeta (a_j)= b_{j_0}.$$
  Let $\rho$ be the identity map, and $\sigma$ be the local inverse of $B_1$ determined by $\sigma (a_1)=a_j $.
Then $$b_{j_0}\in  \zeta \circ  \sigma (a_1) .$$ Noting that $\eta
(a_1)=b_{j_0}$, we deduce that $[\zeta]\circ [\sigma]$ must contain
$[\eta]$. But $$[\rho] \circ [\zeta]=[\zeta]\neq [\eta].$$
Therefore, $[\rho] \circ [\zeta] \neq [\zeta]\circ [\sigma] $,
forcing
$$([\rho] \circ [\zeta](z_2), [\sigma] \circ [\eta] (z_1)) \neq
([\zeta]\circ [\sigma] (z_2), [\eta] \circ [\rho] (z_1)).$$ This
means  that there are two equivalences of admissible local inverses
of $\Upsilon_{B_1,B_2}$, (\ref{loc1}) and (\ref{loc2}),   do not
commute. Then by Theorem \ref{ormain}  $ \mathcal{V}^*(
\Upsilon_{B_1,B_2},\mathbb{D}^2) $ is not abelian.

Now assume $S_{B_1,B_2}$ is connected. That is,  $S_{B_1,B_2}$  has
only one  component  as well as $S_{B_2,B_1}$. In this case  we will
prove that $ \mathcal{V}^*( \Upsilon_{B_1,B_2},\mathbb{D}^2) $ is
abelian. First, it will be shown that the equivalence of a local
inverse  (\ref{loc1}) commutes with  that of (\ref{loc2}).
 Assume $[\zeta]$
 and $[\eta]$ are the only equivalent classes for local solutions of
 $S_{B_1,B_2}$ and $S_{B_2,B_1}$, respectively.  By careful analysis of boundary behaviors of $B_1$ and $B_2,$
 we get
$$[\rho] \circ [\zeta] =\sharp [\rho]\cdot   [\zeta] \quad    \mathrm{and} \quad
 [\sigma] \circ [\eta] =\sharp  [\sigma] \cdot  [\eta].$$
Thus, letting $z_1$ and $z_2$ denote complex variables gives $$([\rho]
\circ [\zeta](z_2), [\sigma] \circ [\eta] (z_1))= \sharp [\rho]\cdot
\sharp  [\sigma] (  [\zeta](z_2),   [\eta] (z_1)).$$ Similarly,
$$([\zeta]\circ [\sigma] (z_2), [\eta] \circ [\rho] (z_1)) =\sharp
[\rho]\cdot \sharp  [\sigma] (  [\zeta](z_2),   [\eta] (z_1)),$$
which gives $$([\rho] \circ [\zeta](z_2), [\sigma] \circ [\eta]
(z_1)) = ([\zeta]\circ [\sigma] (z_2), [\eta] \circ [\rho] (z_1)).$$
Thus, the equivalence of a local inverse  (\ref{loc1}) commutes with the  equivalence of (\ref{loc2}).

As follows, we will prove that two  equivalences of local inverses with
the form (\ref{loc1}) commutes. In fact, since $B_1$ is a finite
Blaschke product,
 by  \cite[Theorem 1.1]{DPW}   $\mathcal{V}^*(B_1,\mathbb{D})$ is abelian.
Since $\mathcal{V}^*(B_1,\mathbb{D})$ is generated by
$\mathcal{E}_{[\rho]}$ where $\rho$ are local inverses of $B_1$, we
have
$$[\rho_1] \circ[\rho_2]  =[\rho_2]\circ[\rho_1], \, \rho_1,\rho_2   \in B_1^{-1}\circ B_1.$$
Similarly,
$$[\sigma_1] \circ [\sigma_2] = [\sigma_2] \circ [\sigma_1],    \, \sigma_1, \sigma_2 \in   B_2^{-1}\circ B_2. $$
Therefore, we have  $$([\rho_1](z_1),[\sigma_1](z_2))\circ
([\rho_2](z_1),[\sigma_2](z_2))=
   ([\rho_2](z_1),[\sigma_2](z_2)) \circ  ([\rho_1](z_1),[\sigma_1](z_2)). $$
In addition, one can prove that any two equivalences of local inverses of the form
(\ref{loc2}) commute with each other. In summary, all admissible local inverses of
$\Upsilon_{B_1,B_2}$ commute with each other under composition.
Therefore, if $S_{B_1,B_2}$ is connected,   $ \mathcal{V}^*(
\Upsilon_{B_1,B_2},\mathbb{D}^2) $ is  abelian.
      The proof is complete.
 $\hfill \square$
 \vskip2mm

Some special cases of Theorem \ref{bla} are of interest. If
$B_1=B_2$, one component of $\mathcal{S}_{B_1,B_2}$ is $\{(z,z):
z\in \mathbb{D}-J\}$, where $J$ is a finite set. Therefore,
$\mathcal{S}_{B_1,B_1}$
  is connected  if and only if  order $B_1$=1.  This immediately gives \cite[Example 6.5]{HZ}. %????

In the case of $B_1\neq B_2$, things are interesting, see the
following.
\begin{cor}  Let $B_1$ and $B_2$
 be two finite Blaschke products.
Write \linebreak $m=\mathrm{order}\, B_1, $ and  $n=\mathrm{order}\, B_2.$ If
$GCD(m,n)=1$,    \label{cor}
 then   $ \mathcal{V}^*( \Upsilon_{B_1,B_2},\mathbb{D}^2) $ is abelian.
  \end{cor}
 \begin{proof}  Recall that  $\mathcal{S}_{B_1,B_2}$  is connected if and only if all local solution for
$\mathcal{S}_{B_1,B_2}$ are equivalent in the sense of analytic
continuation. By Theorem
\ref{bla}, it suffices to show that if   $GCD(m,n)=1$,
 then  all local solutions for $\mathcal{S}_{B_1,B_2}$ are equivalent.  For this, note that the proof of Theorem \ref{bla}
has shown that a local solution for $S_{B_1,B_2}$ admits
   continuation along any curve contained in $\mathbb{T}$. Without loss of generality, $m>n.$
   Let $a_j$ and $b_k$ be chosen as in the proof of Theorem \ref{bla}.
 Suppose $\zeta$ is a local solution satisfying $$\zeta(a_1)=b_1.$$
 Note that for $1\leq j \leq m-1$ and $1\leq k \leq n-1 $, the image of
 the circular arc $\widetilde{a_ja_{j+1}}$ under $B_1$
is the same as that of the circular arc $\widetilde{b_kb_{k+1}}$
under $B_2$. Then we get
$$\widetilde{\zeta}(a_j)=b_j, \, 1\leq j \leq m.$$ where $ \widetilde{\zeta}$ denotes
some analytic continuation along a circular curve  $\gamma$ in
 $\mathbb{T}.$ Letting $\gamma$ go a bit further, and noting $\widetilde{\zeta}(a_m)=b_m$, we have
  $$ \breve{\zeta} (a_1)=b_{m+1} ,$$
where $ \breve{\zeta} $ is also an analytic continuation  of
$\zeta.$ This procedure can be repeated. Since $GCD(m,n)=1,$   for each $k(1\leq k \leq n)$ there exists an analytic
continuation $\eta$ of $\zeta$ such that
$$\eta (a_1)=b_k.$$
Thus all local solutions for $\mathcal{S}_{B_1,B_2}$ is an analytic
continuation of $\zeta$, and hence
   $\mathcal{S}_{B_1,B_2}$ has exactly one component. That is,  $\mathcal{S}_{B_1,B_2}$
is  connected to finish the proof. \end{proof} The inverse of  Corollary \ref{cor} fails
in general. However, the following shows it holds in a special case.

\begin{exam}  Write $B_1(\lambda)=\lambda^k$ and  $B_2(\lambda)=\lambda^l$,  where $\lambda \in \mathbb{D}$ and
$k,l\in\mathbb{Z}_+. $ We claim that  $\mathcal{S}_{B_1,B_2}$ is
connected if and only if $$GCD(k,l)=1.$$   Then by Theorem
\ref{bla}  $\mathcal{V}^*(z_1^k+z_2^l, z_1^{2k}+ z_2^{2l}, \mathbb{D}^2)$
is abelian if and only if  $$GCD(k,l)=1,$$  if and only  \label{kl}
if \,  $\mathcal{V}^*(z_1^k+z_2^l, \mathbb{D}^2)$ is abelian. The
latter statement follows directly from \cite[Theorem 1.1]{DH}.

To finish the proof, we must show that $\mathcal{S}_{B_1,B_2}$ is
connected if and only if $$GCD(k,l)=1.$$ In fact, by the discussion
in the above paragraph, if $GCD(k,l)=1,$ then
$\mathcal{S}_{B_1,B_2}$ is connected.
Now suppose $GCD(k,l)=d_0>1 $. % and write $$k=k' d_0,\,  l=l' d_0.$$
Write $\omega_j= \exp(\frac{j 2\pi i}{k})$ and $$ \varsigma_j=
\exp(\frac{j 2\pi i}{l}),\, j\in \mathbb{Z}_+.$$
 Let $\eta$ denote local solution for $\mathcal{S}_{B_1,B_2}$  satisfying
 $\eta(\omega_1)=\varsigma_1.$    Note that
each local solution for $\mathcal{S}_{B_1,B_2}$ admits analytic
continuation  along any curve in $ \mathbb{C} -\{0\}$.
  All continuations of $\eta$ can be realized by a loop on $\mathbb{T}$.
  Then it is not difficult to see that
  $$[\eta](\omega_1) =\{\varsigma_{1+m d_0}: m\in \mathbb{Z}_+ \}.$$
Clearly, $[\eta](\omega_1) \neq \{  \varsigma_j: j\in \mathbb{Z}_+
\}$. Let  $j_0$  be an integer such that  $$ \varsigma_{j_0} \not\in
[\eta](\omega_1) ,$$ and let $\zeta$ denote the local solution for
$\mathcal{S}_{B_1,B_2}$  satisfying
 $\zeta(\omega_1)=\varsigma_{j_0}.$  Then
 $$[\eta] \neq [\zeta].$$
Thus there are at least two components of $\mathcal{S}_{B_1,B_2}$,
forcing
 $\mathcal{S}_{B_1,B_2}$  not to be
connected as desired.
\end{exam}
 Theorem \ref{bla} fails if   more than two Blaschke products are involved.
 The following shows that the number of variables play an important role here.   %intertwined.

\begin{exam}Let $d\geq 3.$ For $d$ Blaschke products $B_1$, $B_2 \cdots, B_d,$
 let $\Phi(z)$ denote
 $$(\sum_{k=1}^d B_k(z_k) ,  \sum_{k=1}^d B_k(z_k)^2, \cdots, \sum_{k=1}^d B_k (z_k)^d  ).$$
 Then  $ \mathcal{V}^*( \Phi,\mathbb{D}^d) $ is not abelian.

 We will just deal with the case of  $d=3$, and the remaining case  can be  similarly handled.  \label{tri}
Consider  some   solutions for the equation $\Phi(w)=\Phi(z):$
\[    \left\{\begin{array}{ccc}
 B_1(w_1)=B_2(z_2) \\
   B_2(w_2)=B_1(z_1) \\
    B_3(w_3)=B_3(z_3)
\end{array}\right.
\]
Then   $(w_1,w_2,w_3)=(\zeta_{1,2}(z_2),
\zeta_{2,1}(z_1),\zeta_{3,3}(z_3)),$ where $$\zeta_{j,k} \in
B_j^{-1}\circ B_k (1\leq j,k \leq 3) $$ are admissible local
solutions. Also, write
\[    \left\{\begin{array}{ccc}
 B_1(w_1)=B_3(z_3) \\
   B_2(w_2 )=B_2(z_2) \\
    B_3(w_3 )=B_1(z_1)
\end{array}\right.
\]
Then we get some other  admissible local inverses of $\Phi$:
$$(\zeta_{1,3}(z_3), \zeta_{2,2}(z_2),\zeta_{3,1}(z_1)).$$
By taking the compositions of their equivalent classes, we get
$$( [\zeta_{1,2}]\circ [\zeta_{2,2}](z_2),
 [\zeta_{2,1}]\circ [\zeta_{1,3}](z_3), [\zeta_{3,3}]\circ [\zeta_{3,1}](z_1)), $$
 and
$$( [\zeta_{1,3}]\circ [\zeta_{3,3}](z_3),
 [\zeta_{2,2}]\circ [\zeta_{2,1}](z_1), [\zeta_{3,1}]\circ [\zeta_{1,2}](z_2)). $$
Clearly, they are not   equal since the  variables do not
match. Thus not all equivalent classes of  local inverse
  commute  with each other under composition. Therefore,
  by Theorem \ref{ormain} $ \mathcal{V}^*( \Phi,\mathbb{D}^d) $ is not abelian.
 \end{exam}

\section{\label{sect5}General twisted proper maps}

  This section mainly generalize the results in the last section. We will
  focus on two twisted holomorphic proper maps and present
 the proof of   Theorem  \ref{32}.

To begin with, we give a general setting. Let $F=(f_1,\cdots, f_d )$ be a holomorphic function over
a domain on $\mathbb{C}^d$.   Define \begin{equation}\Upsilon_F(z)= (\varphi_1(z), \cdots,
\varphi_d(z)), \label{deff} \end{equation} where $$\varphi_1(z)= \sum_{j=1}^d f_j(z),
\varphi_2(z)=\sum_{j=1}^d f_j(z)^2, \cdots $$
 and
 $ \varphi_d(z)= \sum_{1\leq j\leq d} f_j(z) ^d.$
 Write
 $$\psi_1=\varphi_1, \psi_2(z)=\sum_{1\leq j<k \leq d}  f_j(z)f_k(z),  \cdots $$
 and
 $ \psi_d(z)= \Pi_{1\leq j\leq d} f_j(z)  .$
 Consider  the equation
 $$\Upsilon_F(w)=\Upsilon_F(z);$$ that is, $$(\varphi_1(w), \cdots, \varphi_d(w))=(\varphi_1(z), \cdots, \varphi_d(z)).$$
 This is equivalent to
$$(\psi_1(w), \cdots,\psi_d(w))=(\psi_1(z), \cdots, \psi_d(z)).$$
 Note that
 $$x^d- \psi_1(z) x^{d-1}+\cdots+ (-1)^{d-1} \psi_{d-1}(z) x + (-1)^{d}  \psi_d(z)=\prod_{j=1}^d (x-f_j(z)),$$
 and then $(f_1(w),\cdots, f_d(w))$ is a permutation of $(f_1(z),\cdots, f_d(z)).$
Thus solving the equation
 $\Upsilon_F(w)=\Upsilon_F(z)$
    is equivalent to solving all equations:
     \begin{equation}f_j(w)=f_{\pi (j)}(z),\,  1\leq j\leq d ,  \label{compat}
 \end{equation}
where $\pi$ runs over all permutations of $\{1,\cdots,d\}.$

\vskip2mm

A special case of (\ref{compat}) will be our focus. For simplicity, let  $\Omega$ be a bounded
 domain in $\mathbb{C}^2$. Let
 $\Phi=(\phi_1,\phi_2)$ and $\Psi=(\psi_1,\psi_2)$ be  holomorphic proper maps over $\Omega $ satisfying
$$\Phi(\Omega)=\Psi(\Omega).$$
In addition, we assume both of them are holomorphic on $\overline{\Omega}$.
Write
$$f_1(z)=\phi_1(z_1,z_2), f_2(z)=\phi_2(z_1,z_2),$$
and
$$f_3(z)=\psi_1(z_3,z_4), f_4(z)=\psi_2(z_3,z_4).$$
Put $F=(f_1,\cdots,f_4)$, and  rewrite $\Upsilon_{\Phi,\Psi}=\Upsilon_F$.
 %$$\Upsilon_{\Phi,\Psi}=\Phi_I \circ ( \Phi,\Psi ).$$
To investigate the structure of $\mathcal{V}^*(\Upsilon_{\Phi,\Psi}, \Omega^2)$, we must determine
 all admissible local inverses for $\Upsilon_{\Phi,\Psi}$ on $\Omega^2$. Among them, two admissible local inverses
 are related to
  $\mathcal{S}_\Phi$, $\mathcal{S}_\Psi$, and  $\mathcal{S}_{\Phi, \Psi}$.
  Precisely, let
 $$
 (\Phi(w_1,w_2),\Psi(w_3,w_4))=(\Phi(z_1,z_2),\Psi(z_3,z_4)),  $$ and
$$   (\Phi(w_1,w_2),\Psi(w_3,w_4))=(\Psi(z_3,z_4),\Phi(z_1,z_2)).$$
 We get their solutions $w=(w_1,w_2 , w_3,w_4)$:
 \begin{equation}w=(\sigma_1(z_1,z_2), \sigma_2(z_3,z_4)),\, \sigma_1  \in \Phi^{-1}\circ \Phi,\sigma_2 \in \Psi^{-1}\circ \Psi, \label{4.1}
 \end{equation}
 and
\begin{equation}w=(\eta_1(z_3,z_4), \eta_2(z_1,z_2)) ,\, \eta_1 \in \Phi^{-1}\circ \Psi,\eta_2 \in \Psi^{-1}\circ \Phi. \label{4.2}
 \end{equation}
 By Lemma \ref{admiss}, both (\ref{4.1}) and (\ref{4.2}) give
 admissible local inverses of $\Upsilon_{\Phi,\Psi}$.
\vskip1mm
  It is possible that there exists  admissible local inverse not of the form
   (\ref{4.1}) or (\ref{4.2}). The following is such an example.
   \begin{exam} Let
 $\Omega=\mathbb{D}^2$, put  $\Phi(z_1,z_2)=(B_1(z_1),B_2(z_2))$
and $$\Psi(z_1,z_2)=(B_3(z_1),B_4(z_2)),$$ where $B_1,\cdots,B_4$ are  finite Blaschke products.
 In this case, we have more admissible local inverses of $\Upsilon_{\Phi,\Psi}.$
 In fact, for  a permutation $\pi$  of $\{1,\cdots,4\} $ (\ref{compat}) gives
  $$B_j(w_j)=B_{\pi (j)}(z_{\pi(j)}), $$
  and this gives admissible local inverses
  $$(\rho_1(z_{\pi(1)}), \cdots, \rho_4(z_{\pi(4)} )),\,
  \rho_j \in B_j^{-1} \circ B_{\pi (j)}, 1\leq j \leq 4.$$
Since there are many choices of $\pi$,   $\Upsilon_{\Phi,\Psi} $ has  a lot of admissible local inverses.
\end{exam}

\vskip2mm
  Let $(g_1,g_2,g_3,g_4) $ be a permutation of $(f_1,f_2,f_3,f_4)$.
By (\ref{compat}), we get
$$(f_1(w),f_2(w),f_3(w),f_4(w))=(g_1(z),g_2(z),g_3(z),g_4(z)). $$
Letting $\rho$ be a local inverse of $\Upsilon_{\Phi,\Psi}$ gives
$$(f_1(\rho(z)),f_2(\rho(z)),f_3(\rho(z)),f_4(\rho(z)))=(g_1(z),g_2(z),g_3(z),g_4(z)), z\in \Omega-\mathcal{E} , $$
where $\mathcal{E}$ is a    subset of $\Omega$ with zero Lebesgue measure.
Then by  (\ref{subset}) we get
\begin{equation}\overline{   (f_1,f_2,f_3,f_4)(\Omega^2)} =\overline{ (g_1 ,g_2 ,g_3 ,g_4   )(\Omega^2)} . \label{comp}
 \end{equation}
  The equation  $$(f_1(w),f_2(w),f_3(w),f_4(w))=(g_1(z),g_2(z),g_3(z),g_4(z))  $$
 is called \emph{compatible} if  (\ref{comp}) holds.
If  the only possible compatible equations are
$$(\Phi(w_1,w_2),\Psi(w_3,w_4))=(\Phi(z_1,z_2),\Psi(z_3,z_4)),$$
and $$(\Phi(w_1,w_2),\Psi(w_3,w_4))=(\Psi(z_3,z_4),\Phi(z_1,z_2)),$$
then we call $\Upsilon_{\Phi,\Psi}$ \emph{has no nontrivial  compatible equations.}

 By the above analysis, we have the following result, which tells us that in some cases (\ref{4.1}) and  (\ref{4.2}) exhaust all admissible local inverse
 of  $\Upsilon_{\Phi,\Psi}$ on $\Omega^2$.
\begin{thm} Suppose $\Phi$ and $\Psi $ are two holomorphic proper maps over $\Omega $ such that
$\Phi(\Omega)=\Psi(\Omega)$, and both maps are holomorphic on
$\overline{\Omega}.$ Assume that $\Upsilon_{\Phi,\Psi}$ has no
nontrivial  compatible equation.  \label{53} Then
$\mathcal{V}^*(\Upsilon_{\Phi,\Psi},\Omega^2)$ is generated by
$\mathcal{E}_{[\rho]}$, where $ \rho $ is of the form (\ref{4.1}) or
(\ref{4.2}).
 \end{thm}
\noindent  Under conditions in Theorem \ref{53},
$\mathcal{V}^*(\Upsilon_{\Phi,\Psi},\Omega^2)$ is   trivial if and
only if  $\Phi= \Psi $ and  $\Phi$ is    biholomorphic.
\begin{cor} Under the conditions in Theorem \ref{53},
  $\mathcal{V}^*(\Upsilon_{\Phi,\Psi},\Omega^2)$ is
not $*$-isomorphic to $\mathcal{V}^*( \Phi,\Omega  )\otimes
\mathcal{V}^*( \Psi,\Omega)$.
\end{cor}
  Usually $\Upsilon_{\Phi,\Psi}$   has no
nontrivial  compatible equation. The following provides such an
example.
\begin{exam} Let $\Omega=\mathbb{D}^2,$ and   define
  $$\Phi(z_1,z_2)=(z_1^2+z_2^3, z_1^2 z_2^3) ,\,(z_1,z_2) \in \Omega, $$ and
   $$\Psi(z_1,z_2)=(z_1+z_2,z_1z_2),\,(z_1,z_2) \in \Omega .$$
It is clear that $\Phi(\Omega)=\Psi(\Omega)$.

First, we state that there is no nontrivial compatable equations for $ \Upsilon_{\Phi,\Psi}$. For this, consider \[    \left\{\begin{array}{cccc}
  w_1^2+w_2^3=z_1^2+z_2^3 ,\\
  w_1^2w_2^3=z_3z_4,  \\
    w_3+w_4= z_3+z_4 ,\\
w_3w_4=z_1^2 z_2^3.
\end{array}\right.
\]
Define $\pi_1(z)=z_1$, $\pi_{1,2}(z)=(z_1,z_2)$ for $z\in \mathbb{C}^4$, and etc. By (\ref{comp}) we have
$$\pi_1 \overline{   (f_1,f_2,f_3,f_4)(\Omega^2)} =\pi_1 \overline{ (g_1 ,g_2 ,g_3 ,g_4   )(\Omega^2)} , $$
 \begin{equation}\pi_{1,2} \overline{   (f_1,f_2,f_3,f_4)(\Omega^2)} =\pi_{1,2} \overline{ (g_1 ,g_2 ,g_3 ,g_4   )(\Omega^2)} , \label{e12} \end{equation}
and so on. By (\ref{e12}) we would have
$$\overline{\{(w_1^2+w_2^3,w_1^2 w_2^3): w_1,w_2 \in \mathbb{D}\} } = 2\overline{ \mathbb{D}}\times \overline{\mathbb{D}}, $$
but  this is impossible.
Similarly,  any equations of the form
\[    \left\{\begin{array}{cccc}
  w_1^2+w_2^3=z_3+z_4 ,\\
  w_1^2w_2^3=z_1^2z_2^3  ,  \\
    \cdots  ,\\
 \cdots
\end{array}\right.
\]
is not compatible. The remaining cases can be handled in a similar way. Thus,  there is no nontrivial compatable equations for $ \Upsilon_{\Phi,\Psi}$.

As follows, we will determine all admissible local inverses of  $ \Upsilon_{\Phi,\Psi}$. For this,
it suffices to consider compatible equations. First write
  \[    \left\{\begin{array}{cccc}
  w_1^2+w_2^3=z_1^2+z_2^3 ,\\
  w_1^2w_2^3=z_1^2 z_2^3,  \\
    w_3+w_4= z_3+z_4 ,\\
w_3w_4=z_3z_4.
\end{array}\right.
\]
The corresponding local inverse is
$$(\rho(z_1,z_2), \sigma(z_3,z_4)),$$
where $\rho$ has the form
 $$ (\sqrt{z_2^3} , \sqrt[3]{z_1^2}) \quad  \mathrm{or} \quad ( \omega_1z_1,\omega_2z_2),
  \quad \mathrm{with} \quad \omega_1^2=\omega_2^3=1, $$
and
$$ \sigma(z_3,z_4)=(z_3,z_4) \quad  \mathrm{or} \quad \sigma(z_3,z_4)=(z_4,z_3). $$
It remains to deal with the equations
  \[    \left\{\begin{array}{cccc}
  w_1^2+w_2^3= z_3+z_4,\\
  w_1^2w_2^3=z_3z_4,  \\
    w_3+w_4= z_1^2+z_2^3 ,\\
w_3w_4=z_1^2 z_2^3.
\end{array}\right.
\]
Then we get the local inverses
$$(  \sqrt{z_3}, \sqrt[3]{z_4},  z_1^2, z_2^3),\, (  \sqrt{z_3}, \sqrt[3]{z_4}, z_2^3, z_1^2  )$$
and $$(   \sqrt[3]{z_4},\sqrt{z_3},  z_1^2, z_2^3),\, ( \sqrt[3]{z_4}, \sqrt{z_3},
 z_2^3, z_1^2  ) . $$
All the above local inverses are admissible. This shows that
$\mathcal{V}^*(\Upsilon_{\Phi,\Psi},\Omega^2)$ has rich structures.
\end{exam}
  Compared with Theorem \ref{53}, Theorem \ref{32} is of special interest.
   We now comes to its  proof.
\vskip2mm
\noindent \textbf{Proof of   Theorem \ref{32}.}  Suppose $\Phi$ and $\Psi $ are two holomorphic proper maps over $\Omega $ and both maps are
 holomorphic on $\overline{\Omega} $. We must determine all admissible local inverses of
$\Upsilon_{\Phi,\Psi}$.
Since $\Upsilon_{\Phi,\Psi}$ has no nontrivial  compatible equation. It reduces to  two cases of (\ref{compat}):
\begin{equation}(\Phi(w_1,w_2),\Psi(w_3,w_4))=(\Phi(z_1,z_2),\Psi(z_3,z_4))\label{fg}
 \end{equation}
and
\begin{equation}(\Phi(w_1,w_2),\Psi(w_3,w_4))=(\Psi(z_3,z_4),\Phi(z_1,z_2))\label{gf}
 \end{equation}
   If there were an admissible
 local solution for  (\ref{gf}), then by Remark \ref{eq} and (\ref{subset})
  $$\overline{\Phi(\Omega)\times \Psi(\Omega)} =\overline{ \Psi(\Omega)\times \Phi(\Omega)},$$
  forcing $ \overline{\Phi(\Omega)} =\overline{\Psi(\Omega)}$. This is a contradiction.
  Therefore, there is no admissible
 local solution for  (\ref{gf}).

It remains  to deal with  (\ref{fg}). It is clear that each admissible local solution $\eta$ of  (\ref{fg})
is exactly of the form  $(\rho(z_1,z_2),\sigma(z_3,z_4)),$  where
$\rho$ and $\sigma$ are  admissible local inverses of $\Phi$ and
$\Psi$ in  $\Omega$, respectively. Since $\Phi$ is a holomorphic
proper maps over $\Omega $, all of its local inverses are
admissible. Theorem \ref{ormain}  shows that
$\mathcal{V}^*(\Phi,\Omega)$ is generated by $\mathcal{E}_{\rho}$,
where $\rho$ run over all admissible local inverses of $\Phi$.
The same is true for $\mathcal{V}^*(\Psi,\Omega)$. Then by
setting
$$\mathcal{E}_{[\rho(z_1,z_2),\sigma(z_3,z_4)]} \mapsto \mathcal{E}_{[\rho]}\otimes  \mathcal{E}_{[\sigma]},$$
we obtain a $*$-isomorphism between  $\mathcal{V}^*(\Upsilon_{\Phi,\Psi}, \Omega^2)  $
and $\mathcal{V}^*(\Phi,\Omega)\otimes \mathcal{V}^*(\Psi,\Omega),$ finishing the proof of   Theorem \ref{32}.
 $\hfill \square $
\vskip2mm

 Both Theorems \ref{53} and \ref{32}   have natural generalizations
  to the case of $\Omega$ being a bounded domain in
 $\mathbb{C}^d.$  The proofs are just the same. Also, one can define a map  like $\Upsilon_{\Phi,\Psi}$
based on three or more holomorphic proper maps on $\Omega$,
 and there is no essential difference.

  Proposition \ref{ffgg} follows essentially from Theorem  \ref{32} and its proof is presented as below.
\vskip1mm
\noindent \textbf{Proof of   Proposition \ref{ffgg}.} Suppose   $f$ and $g$ are holomorphic over $\overline{\mathbb{D}}.$ Following the proof of Theorem \ref{32},
 one obtains a $*$-isomorphism between  $\mathcal{V}^*(\Phi_{f,g}, \mathbb{D}^2)  $
and $\mathcal{V}^*(f,\mathbb{D})\otimes
\mathcal{V}^*(g,\mathbb{D}).$
  % Theorem \ref{ormain},  which implies that if $h  $ is holomorphic over $\overline{\mathbb{D}},$  then
%$\mathcal{V}^*(h,\mathbb{D})$ is generated by $\mathcal{E}_{\rho}$,
%where $\rho$ run over all admissible local inverses of $h$.

Since    $f$ is holomorphic over $\overline{\mathbb{D}},$  by
Thomson's theorem \cite{T1}   there is a finite Blaschke product
$B_f$ such that
$$\mathcal{V}^*(f,\mathbb{D})=\mathcal{V}^*(B_f,\mathbb{D}).$$
Recall that for each finite Blaschke product $B $,   $
\mathcal{V}^*(B ,\mathbb{D})$ is abelian \cite[Theorem 1.1]{DPW}.
Then so is  $\mathcal{V}^*(f,\mathbb{D})$, as well as
$\mathcal{V}^*(g,\mathbb{D})$. Therefore, the von Neumann algebra
$\mathcal{V}^*(f,\mathbb{D})\otimes \mathcal{V}^*(g,\mathbb{D}) $
is abelian, and hence $\mathcal{V}^*(\Phi_{f,g}, \mathbb{D}^2)  $ is
abelian.
 $\hfill \square$

\vskip2mm
 To end this section, we construct  an   example  which essentially arises
from elementary symmetric polynomials.
\begin{exam}Suppose   $\Omega$ is a domain in $\mathbb{C}^d$.
  Recall that the complete homogeneous symmetric polynomial  $h_k$ of degree $k$ in $d$ variables $z_1,\cdots,z_d$,
   is the sum of all monomials of total degree $k $ in the variables. Precisely,
   $$h_k(z)=\sum_{1\leq j_1\leq j_2 \leq \cdots \leq  j_k \leq d} z_{j_1}\cdots z_{j_k}.$$
   Write
    $$\Phi_{I}(z)=(z_1+\cdots+z_d, \sum_{j<k}z_jz_k,\cdots, \Pi_{j=1}^d z_j),$$
 $$\Phi_{II}(z)=(z_1+\cdots+z_d, \sum_{j =1}^d z_j^2 ,\cdots, \sum_{j=1}^d z_j^d),$$
  and $$\Phi_{III}(z)=(h_1(z), h_2(z),\cdots, h_d(z)).$$
Note that both $\Phi_{I}$ and $\Phi_{II}$ contain elementary
symmetric polynomials.  \label{symm} There are a polynomial $p$ and
$q$ such that
$$\Phi_{I}=p\circ \Phi_{II} \quad \mathrm{ and}  \quad \Phi_{II}=q\circ \Phi_{I}.$$
Thus, $M_{\Phi_{I}} $ and  $M_{\Phi_{II}}$ have the same joint
reducing subspaces, and then
$$ \mathcal{V}^*(\Phi_{I},\Omega)=\mathcal{V}^*(\Phi_{II},\Omega).$$
Since every symmetric polynomial can be expressed as a polynomial
expression
 in %complete homogeneous symmetric polynomials
$h_1,\cdots,h_d$, by same reasoning we  get
$$ \mathcal{V}^*(\Phi_{I},\Omega)=\mathcal{V}^*(\Phi_{III},\Omega).$$
Furthermore,  if $\Omega$ is invariant under $S_d$, the permutation
group, then by the comments below \cite[Proposition 6.6]{HZ}  we
have
$\mathcal{V}^*(\Phi_{I},\Omega) \cong \mathcal{L}(S_d).$ % where $S_d$ denotes the permutation group.
For example,  put $\Omega= \mathbb{D}^d$ or $\mathbb{B}_d.$ In
particular, we have $$
\mathcal{V}^*(z_1+z_2,z_1z_2,\mathbb{D}^2)=\mathcal{V}^*(z_1+z_2,z_1^2+z_2^2,\mathbb{D}^2)\cong
\mathcal{L}(S_2).$$
  \end{exam}
%Note that  $\Phi_I $, $\Phi_{II} $ and $\Phi_{III} $ are holomorphic
%proper maps on $\Omega$ if $\Omega$ is
% invariant under $S_d$. To see this, note that  $\Phi_I(w)=\Phi_I(z)$ if
%and only if $\Phi_{II}(w)= \Phi_{II}(z) $;  if and only if
%$\Phi_{III}(w)=\Phi_{III}(z)$. By  \cite[Proposition 6.6]{HZ} %?????
% these hold  if and only if  there is a permutation $\rho$ satisfying
%$$w=\rho(z).$$ Since $\Omega$ is
% invariant under $S_d$,  $\rho $ is in $\mathrm{Aut }(\Omega)$. That is, all local inverses of
% $\Phi_I$ are in $\mathrm{Aut }(\Omega)$. By Lemma \ref{propcri} $\Phi_I $ is a proper map on $\Omega$, as
% well as $\Phi_{II}$.

%  Let  $ \Theta $  be one of the map $\Phi_I,$
%$\Phi_{II} $ and $\Phi_{III} $ defined in
By same reasoning in Example \ref{symm}, we get the following.
\begin{prop}   Suppose   $H:\Omega\to \mathbb{C}^d$ is a bounded
holomorphic function. Then  \label{unit}
$$\mathcal{V}^*(\Phi_I\circ H, \Omega)=\mathcal{V}^*(\Phi_{II}\circ H, \Omega)=\mathcal{V}^*(\Phi_{III}\circ H, \Omega).$$
  \end{prop}
 In defining the map $\Upsilon$, we have made use of the map $\Phi_{II}.$
Proposition \ref{unit} shows that in our theorems and propositions, one  can just replace $\Phi_{II} $ with
$\Phi_{I} $ or $\Phi_{III} $ without changing anything.

\vskip3mm \noindent{Center of Mathematics, Chongqing University,
Chongqing,  401331, China,
   E-mail: chongqingmath@gmail.com

\noindent    Department of Mathematics, East China University of
Science and Technology, Shanghai, 200237, China, E-mail:
hshuang@ecust.edu.cn
 }
\end{document}